\documentclass[reqno]{amsart}
\usepackage{enumerate,color}

\newtheorem{theorem}{Theorem}[section]
\newtheorem{lemma}[theorem]{Lemma}
\newtheorem{proposition}[theorem]{Proposition}

\theoremstyle{definition}

\renewcommand{\Re}{\operatorname{Re}}

\newcommand{\D}{\operatorname{\{s\in\mathbb{C}:1/2<\Re(s)<1\}}}
\newcommand{\meas}[1]{\liminf_{T\to\infty}\frac{1}{T}\operatorname{meas}\left\{\tau\in [0,T]: #1\right\}}
\newcommand{\measlog}[1]{\liminf_{T\to\infty}\frac{1}{T}
\operatorname{meas}\left\{\tau\in [2,T]: #1\right\}}

\begin{document}
\title{Joint universality for dependent $L$-functions}
\author{{\L}ukasz Pa\'nkowski}
\address{Faculty of Mathematics and Computer Science, Adam Mickiewicz University, Umultowska 87, 61-614 Pozna\'{n}, Poland, and Graduate School of Mathematics, Nagoya University, Nagoya, 464-8602, Japan}
\email{lpan@amu.edu.pl}

\thanks{The author is an International Research Fellow of the Japan Society for the Promotion of Science and the work was partially supported by (JSPS) KAKENHI grant no. 26004317 and the grant no. 2013/11/B/ST1/02799 from the National Science Centre.}

\subjclass[2010]{Primary: 11M41}

\keywords{joint universality, uniform distribution, $L$-function}

\begin{abstract}
We prove that, for arbitrary Dirichlet $L$-functions $L(s;\chi_1),\ldots,\break L(s;\chi_n)$ (including the case when $\chi_j$ is equivalent to $\chi_l$ for $j\ne k$), suitable shifts of type $L(s+i\alpha_jt^{a_j}\log^{b_j}t;\chi_j)$ can simultaneously approximate any given analytic functions on a simply connected compact subset of the right open half of the critical strip, provided the pairs $(a_j,b_j)$ are distinct and satisfy certain conditions. Moreover, we consider a discrete analogue of this problem where $t$ runs over the set of positive integers.
\end{abstract}

\maketitle

\section{Introduction}

In 1975, Voronin \cite{V} discovered a universality property for the Riemann zeta-function $\zeta(s)$, namely he proved that for every compact set $K\subset\D$ with connected complement, any non-vanishing continuous function $f(s)$ on $K$, analytic in the interior of $K$, and every $\varepsilon>0$ we have
\begin{equation}\label{eq:SR}
\meas{\max_{s\in K}|\zeta(s+i\tau)-f(s)|<\varepsilon}>0,
\end{equation}
where $\operatorname{meas}\{\cdot\}$ denotes the real Lebesgue measure. Moreover, in 1977,  Voronin \cite{V2} proved the so-called joint universality which, roughly speaking, states that any collection
of Dirichlet L-functions associated with non-equivalent characters can simultaneously and uniformly
approximate non-vanishing analytic functions in the above sense. In other words, in order to approximate a collection of non-vanishing continuous functions on some compact subset of $\D$ with connected complement, which are analytic in the interior, it is sufficient to take twists of the Riemann zeta function with non-equivalent Dirichlet characters. The requirement that characters are pairwise non-equivalent is necessary, since it is well-known that Dirichlet $L$-functions associated with equivalent characters differ from each other by a finite product and, in consequence, one cannot expect joint universality for them. This idea was extended by \v{S}le\v{z}evi\v{c}ien\.{e} \cite{Sl} to certain $L$-functions associated with multiplicative functions, by Laurin\v{c}ikas and Matsumoto \cite{LM} to $L$-functions associated with newforms twisted by non-equivalent characters, and by Steuding in \cite[Section 12.3]{S} to a wide class of $L$-functions with Euler product, which can be compared to the well-known Selberg class. Thus, one possible way to approximate a collection of analytic functions by a given $L$-function is to consider its twists with sufficiently many non-equivalent characters.

Another possibility to obtain a joint universality theorem by considering only one $L$-function was observed by Kaczorowski, Laurin\v{c}ikas and Steuding \cite{KLS}. They introduced the Shifts Universality Principle, which says that for every universal $L$-function $L(s)$, in the Voronin sense, and any distinct real numbers $
\lambda_1,\ldots,\lambda_n$ the functions $L(s+i\lambda_1),\ldots,L(s+i\lambda_n)$ are jointly universal for any compact set $K\subset\D$ satisfying $K_k\cap K_j=\emptyset$ for $1\leq k\ne j\leq n$, where $K_j=\{s+\lambda_j:s\in K\}$.

Next, one can go further and ask if there exists any other transformation of the Riemann zeta function, or a given $L$-function in general, to approximate arbitrary given collection of analytic functions. For example, we might consider a $L$-function, a compact set $K\subset\D$ with connected complement, non-vanishing continuous functions $f_1,\ldots,f_n$ on $K$, analytic in the interior of $K$, and ask for functions $\gamma_1,\ldots,\gamma_n\colon \mathbb{R}\to\mathbb{R}$ satisfying
\begin{equation}\label{eq:mainQuestion}
\meas{\max_{s\in K}|L(s+i\gamma_j(\tau))-f_j(s)|<\varepsilon}>0.
\end{equation}

Obviously, the Shifts Universality Principle gives a partial (under some restriction on $K$) answer for the simplest case when $\gamma_j(\tau) = \tau + \lambda_j$. The consideration for other linear functions $\gamma_j(\tau)=a_j\tau+b_j$ might be restricted, without loss of generality, to the case when $\gamma_j(\tau)=a_j\tau$, which was firstly investigated by Nakamura \cite{N09}, \cite{N11}. He proved that \eqref{eq:mainQuestion} holds, provided $\gamma_j(\tau)=a_j\tau$ with algebraic real numbers $a_1,\ldots,a_n$ linearly independent over $\mathbb{Q}$. Although Nakamura's result is the best known result concerning all positive integers $n$, the case $n=2$ is already much better understood, and from the work of the author and Nakamura (see \cite{N09, NP, P09, P16} we know that \eqref{eq:mainQuestion} holds if $\gamma_1(\tau)=a_1\tau$, $\gamma_2(\tau)=a_2\tau$ with non-zero real $a_1,a_2$ satisfying $a_1\ne \pm a_2$.

The main purpose of the paper is to find other example of functions $\gamma_1,\ldots,\gamma_n$ such that $\eqref{eq:mainQuestion}$ holds. Our approach is rather general and bases on Lemma \ref{lem:finiteProd} and Lemma \ref{lem:2kmomentsApp}, which are stated in the general form. However, we focus our attention only on the case when $\gamma_j(t)=\alpha_j t^{a_j}(\log t)^{b_j}$. The consideration when $a_j=a_k$ and $b_j=b_k$ for some $j\ne k$ is very similar to the already quoted work of the author and Nakamura for linear functions $\gamma(t)$ and requires an extra assumption on $\alpha_j,\alpha_k$, so in the sequel we assume that $a_j\ne a_k$ or $b_j\ne b_k$ for $j\ne k$. Moreover, for the sake of simplicity we will restrict ourselves only to Dirichlet $L$-functions, but it should be noted that our approach can be easily generalized to other $L$-functions (as in \cite{S}), at least in the strip where a good estimate for the second moment is known. On the other hand, we consider any collection of Dirichlet $L$-functions as an input instead of a single $L$-function. Hence, the following theorem gives an easy way how to approximate any collection of analytic functions by taking some shifts of any $L$-functions, even equal or dependent. 

\begin{theorem}\label{thm:univNonInt}
Assume that $\chi_1,\ldots,\chi_n$ are Dirichlet characters, $\alpha_1,\ldots,\alpha_n\in\mathbb{R}$, $a_1,\ldots,a_n$  non-negative real numbers and $b_1,\ldots,b_n$ such that
\[
b_j\in \begin{cases}
\mathbb{R}&\ \text{ if $a_j\not\in\mathbb{Z}$};\\
(-\infty,0]\cup (1,+\infty)&\ \text{ if $a_j\in\mathbb{N}$},
\end{cases}
\]
and $a_j\ne a_k$ or $b_j\ne b_k$ if $k\ne j$. Moreover, let $K\subset\D$ be a compact set with connected complement, $f_1,\ldots,f_n$ be non-vanishing continuous functions on $K$, analytic in the interior of $K$. Then, for every $\varepsilon>0$, we have
\begin{equation}\label{eq:mainNonInt}
\measlog{\max_{1\leq j\leq n}\max_{s\in K}|L(s+i\alpha_j\tau^{a_j}\log^{b_j}\tau;\chi_j)-f_j(s)|<\varepsilon}>0.
\end{equation}
\end{theorem}

Next, let us consider the so-called discrete universality, which means that $\tau$ runs over the set of positive integers. It is an interesting problem, since usually discrete universality requires a special care for some $\alpha_j$. For example (see \cite{B} and \cite{R}) if $\gamma_1(k)=\alpha_1k$ and $n=1$, then the case when $\exp(2\pi k/\alpha_1)\in\mathbb{Q}$ for some integer $k$ is more subtle, since the set $\{\frac{\alpha_1\log p}{2\pi}:p\in\mathbb{P}\}\cup\{1\}$ is not linearly independent over $\mathbb{Q}$, which plays a crucial role in the proof. The case $n\geq 2$ for Dirichlet $L$-functions associated with non-equivalent characters and $\gamma_j(k)=\alpha_j k$ was investigated by Dubickas and Laurin\v{c}ikas in \cite{DL}, where they proved discrete joint universality under the assumption that
\begin{equation}\label{eq:linIndwith1}
\left\{\alpha_j\frac{\log p}{2\pi}:p\in\mathbb{P},\ j=1,2,\ldots,n\right\}\cup\{1\}\qquad\text{is linearly independent over $\mathbb{Q}$}.
\end{equation}
Moreover, very recently Laurin\v{c}ikas, Macaitien\.{e} and \v{S}iau\v{c}i\={u}nas \cite{LMS} showed that, for $\gamma_j(k)=\alpha_jk^a$ with $a\in (0,1)$, Dirichlet $L$-functions associated with non-equivalent characters are discretely jointly universal, provided
\begin{equation}\label{eq:linInd}
\left\{\alpha_j\frac{\log p}{2\pi}:p\in\mathbb{P},\ j=1,2,\ldots,n\right\}\qquad\text{is linearly independent over $\mathbb{Q}$}.
\end{equation}

Inspiring by their considerations we shall prove the following discrete version of Theorem \ref{thm:univNonInt}.

\begin{theorem}\label{thm:univDiscrete}
Assume that $\chi_1,\ldots,\chi_n$ are Dirichlet characters, $\alpha_1,\ldots,\alpha_n\in\mathbb{R}$, $a_1,\ldots,a_n$  non-negative real numbers and $b_1,\ldots,b_n$ such that
\[
b_j\in \begin{cases}
\mathbb{R}&\ \text{ if $a_j\not\in\mathbb{Z}$};\\
(-\infty,0]\cup (1,+\infty)&\ \text{ if $a_j\in\mathbb{N}$},
\end{cases}
\]
and $a_j\ne a_k$ or $b_j\ne b_k$ if $k\ne j$. Moreover, let $K\subset\D$ be a compact set with connected complement, $f_1,\ldots,f_n$ be non-vanishing continuous functions on $K$, analytic in the interior of $K$. Then, for every $\varepsilon>0$, we have
\begin{equation}\label{eq:mainDiscrete}
\liminf_{N\to\infty}\frac{1}{N}\sharp\left\{2\leq k\leq N: \max_{1\leq j\leq n}\max_{s\in K}|L(s+i\alpha_jk^{a_j}\log^{b_j}k;\chi_j)-f_j(s)|<\varepsilon\right\}>0.
\end{equation}
\end{theorem}

It should be noted that Theorem \ref{thm:univDiscrete} (as well as Theorem \ref{thm:univNonInt}) might be formulated in a slightly more general form where instead of the assumption on $a_j,b_j$ we assume that the sequence
\begin{equation}\label{eq:unifDist}
\left(\gamma_j(k)\frac{\log p}{2\pi}:j=1,2,\ldots,n,\ p\in A\right)
\end{equation}
is uniformly distributed (resp. continuous uniformly distributed) modulo $1$ for every finite set $A\subset\mathbb{P}$. 

Let us recall that the sequence $(\omega_1(k),\ldots,\omega_n(k))_{k\in\mathbb{N}}$ is \emph{uniformly distributed mod $1$} in $\mathbb{R}^n$ if for every $\alpha_j,\beta_j$, $j=1,2,\ldots,n$, with $0\leq \alpha_j<\beta_j\leq 1$ we have
\[
\lim_{T\to\infty}\frac{1}{N}\sharp\left\{1\leq k\leq N: \{\omega_j(k)\}\in[\alpha_j,\beta_j]\right\} = \prod_{j=1}^n (\beta_j-\alpha_j),
\]
where $\{x\} = x-[x]$. Similarly, we say that the curve $\omega(\tau): [0,\infty]\to\mathbb{R}^n$ is \emph{continuously uniformly distributed mod $1$} in $\mathbb{R}^n$ if for every $\alpha_j,\beta_j$, $j=1,2,\ldots,n$, with $0\leq \alpha_j<\beta_j\leq 1$ we have
\[
\lim_{T\to\infty}\frac{1}{T}\operatorname{meas}\left\{\tau\in (0,T]: \{\omega(\tau)\}\in[\alpha_1,\beta_1]\times\cdots\times[\alpha_n,\beta_n]\right\} = \prod_{j=1}^n (\beta_j-\alpha_j),
\]
where $\{(x_1,\ldots,x_n)\}:=(\{x_1\},\ldots,\{x_n\})$.

One can easily notice that Weyl's criterion (see \cite[Theorem 6.2 and Theorem 9.2]{KN}) shows that \eqref{eq:linIndwith1} and \eqref{eq:linInd} imply that \eqref{eq:unifDist} is (continuous) uniformly distributed mod $1$. Thus, our approach allows to generalize the result of Dubickas and Laurin\v{c}ikas, and the result due to Laurin\v{c}ikas, Macaitien\.{e} and \v{S}iau\v{c}i\={u}nas to more general functions than $\gamma_j(t)=\alpha_jt^a$ with $a\in (0,1]$.

\section{Approximation by finite product}

Essentially we shall follow the original proof of Voronin's result, which, roughly speaking, might be divided into two parts. The first one relies mainly on uniform distribution mod $1$ of the sequence of numbers $\gamma_j(t)\tfrac{\log p}{2\pi}$ (or independece of $p^{i\gamma_j(t)}$) and deals with approximation of any analytic function by shifts of a truncated Euler product. The second one deals with an application of the second moment of $L$-functions to approximate a truncated Euler product by a corresponding $L$-function in the mean-square sense. 

In this section, we shall focus on the first part. In order to do this, for a Dirichlet character $\chi$, a finite set of primes $M$ and real numbers $\theta_p$ indexed by primes, we put
\[
L_{M}(s,(\theta_p);\chi) = \prod_{p\in M}\left(1-\frac{\chi(p)e(-\theta_p)}{p^s}\right)^{-1},
\]
where, as usual, $e(t)=\exp(2\pi i t)$. Note that for $\sigma>1$ we have $L_\mathbb{P}(s,\overline{0};\chi)=L(s,\chi)$, where $\overline{0}$ denotes the constant sequence of zeros and $\mathbb{P}$ the set of all prime numbers.

\begin{lemma}\label{lem:finiteProd}
Assume that the functions $\gamma_j:\mathbb{R}\to\mathbb{R}$, $1\leq j\leq n$, are such that the curve
\[
\gamma(\tau)=\left(\left(\gamma_1(\tau)\frac{\log p}{2\pi}\right)_{p\in M_1},\ldots,\left(\gamma_n(\tau)\frac{\log p}{2\pi}\right)_{p\in M_n}\right)
\]
is continuously uniformly distributed mod $1$ in $\mathbb{R}^{\sum_{1\leq j\leq n}\sharp M_j}$ for any finite sets of primes $M_j$, $1\leq j\leq n$. Moreover, let $\chi_1,\ldots,\chi_n$ be arbitrary Dirichlet characters, $K\subset\{s\in\mathbb{C}:1/2<\sigma<1\}$ be a compact set with connected complement and $f_1,\ldots,f_n$ continuous non-vanishing function on $K$, which are analytic in the interior of $K$. Then, for every $\varepsilon>0$, there is $v>0$ such that for every $y>v$ we have
\[
\operatorname{meas}\left\{\tau\in [2,T]: \max_{1\leq j\leq n}\max_{s\in K}\left|L_{\{p:p\leq y\}}(s+i\gamma_j(\tau),\overline{0};\chi_j) - f_j(s)\right|<\varepsilon\right\}>cT
\]
with suitable constant $c>0$, which does not depend on $y$.
\end{lemma}

Before we give a proof of the above result, let us recall the following crucial result on approximation any analytic function by a truncated Euler product twisted by a suitable sequence of complex numbers from the unit circle.

We call an open and bounded subset $G$ of $\mathbb{C}$ \textit{admissible}, when for every $\varepsilon > 0$ the set $G_\varepsilon = \{s\in\mathbb{C} : |s-w| < \varepsilon \textrm{ for certain }w\in G \}$ has connected complement. 

\begin{lemma}\label{lem:denseness}
For every Dirichlet character $\chi$, an admissible domain $G$ such that $\overline{G}~\subset \{s\in\mathbb{C}\;:\;\frac{1}{2}<\Re(s)<1\}$, every analytic and non-vanishing function $f$ on the closure $\overline{G}$, every finite set of primes $\mathcal{P}$, there exist $\theta_p\in\mathbb{R}$ indexed by primes and a sequence of finite sets $M_1\subset M_2\subset...$ of primes such that $\bigcup_{k=1}^\infty M_k=\mathbb{P}\setminus \mathcal{P}$ and, as $k\to\infty$,
$$L_{M_k}(s,(\theta_p)_{p\in M_k};\chi)\longrightarrow f(s)\quad \textrm{uniformly in}\ \overline{G}.$$
\end{lemma}
\begin{proof}
This is Lemma 7 in \cite{KK}.
\end{proof}

\begin{proof}[Proof of Lemma \ref{lem:finiteProd}]
By Mergelyan's theorem we can assume, without loss of generality, that the $f_j$'s are polynomials. Then we can find an admissible set $G$ such that $K\subset G\subset\overline{G}\subset\{s\in\mathbb{C}:1/2<\sigma<1\}$  and each $f_j$ is analytic nonvanishing on $\overline{G}$. Therefore by Lemma \ref{lem:denseness} with $\mathcal{P}=\emptyset$, there exist real numbers $\theta_{jp}$ for $p\in\mathbb{P}$, $1\leq j\leq n$ such that, for any $z>0$ and $\varepsilon>0$, there are finite sets of primes $M_1,\ldots,M_n$ such that $\{p:p\leq z\}\subset M_j$ for every $j=1,2,\dots,r$ and
\begin{equation}\label{eq:fromDenseness0}
\max_{1\leq j\leq n}\max_{s\in\overline{G}}\left|L_{M_j}(s,(\theta_{jp})_{p\in M_j};\chi_j)-f_j(s)\right|<\frac{\varepsilon}{2}.
\end{equation}
Now, let 
\[
\mathcal{D}:=\{\overline{\omega}=(\omega_{jp})_{p\in Q}^{1\leq j\leq n}:\max_{1\leq j\leq n}\max_{p\in M_j}\Vert\omega_{jp}-\theta_{jp}\Vert<\delta\},
\]
where $Q=\{p:p\leq y\}\supset \bigcup_{1\leq j\leq n}M_j$ and $\delta>0$ is sufficiently small such that
\[
\max_{1\leq j\leq n}\max_{s\in\overline{G}}\left|L_{M_j}(s,(\omega_{jp});\chi_j)-L_{M_j}(s,(\theta_{jp});\chi_j)\right|<\frac{\varepsilon}{2},
\]
provided $(\omega_{jp})\in\mathcal{D}$.

Our assumption on $\gamma(\tau)$ implies that the set $A$ of real $\tau\geq 2$ satisfying
\[
\max_{1\leq j\leq n}\max_{p\in M_j}\left\Vert\gamma_j(\tau)\frac{\log p}{2\pi}-\theta_{jp}\right\Vert<\delta
\]
has a positive density equal to the Jordan measure $m(\mathcal{D})$ of $\mathcal{D}$. Moreover, we have
\begin{equation}\label{eq:fromDenseness}
\max_{1\leq j\leq n}\max_{s\in\overline{G}}\left|L_{M_j}(s+i\gamma_j(\tau),\overline{0};\chi_j)-f_j(s)\right|<\varepsilon,\qquad\text{for $\tau\in A$}.
\end{equation}

Now, let us define $A_T=A\cap [2,T]$ and
\[
I_j = \frac{1}{T}\int_{A_T}\left(\iint_{G}\left|L_Q(s+i\gamma_j(\tau),\overline{0};\chi_j)-L_{M_j}(s+i\gamma_j(\tau),\overline{0};\chi_j)\right|^2d\sigma dt\right)d\tau.
\]
Since
\[
I_j = \frac{1}{T}\int_{A_T}\left(\iint_{G}\left|L_Q(s,(\gamma_j(\tau)\frac{\log p}{2\pi});\chi_j)-L_{M_j}(s,(\gamma_j(\tau)\frac{\log p}{2\pi});\chi_j)\right|^2d\sigma dt\right)d\tau
\]
and $\gamma(\tau)$ is continuously uniformly distributed mod $1$, we obtain (see Lemma A.8.3 in \cite{KV})
\begin{align*}
&\lim_{T\to\infty} \frac{1}{T}\int_{A_T}\left|L_Q(s,(\gamma_j(\tau)\frac{\log p}{2\pi});\chi_j)-L_{M_j}(s,(\gamma_j(\tau)\frac{\log p}{2\pi});\chi_j)\right|^2 d\tau\\
&\qquad\qquad = \idotsint\displaylimits_{\mathcal{D}}|L_{M_j}(s,\overline{\omega};\chi_j)|^2|L_{Q\setminus M_j}(s,\overline{\omega};\chi_j)-1|^2d\omega\\
&\qquad\qquad = \left(\max_{s\in\overline{G}}|f(s)|^2+\varepsilon\right)m(\mathcal{D})\int_0^1\cdots\int_0^1|L_{Q\setminus M_j}(s,\overline{\omega};\chi_j)-1|^2\prod_{p\in Q\setminus M_j}d\omega_{jp}.
\end{align*}
Therefore, since $Q\setminus M_j$ contains only primes greater that $z$, we have 
\[
I_j<\frac{\sqrt{\pi}\operatorname{dist}(\partial G,K)m(\mathcal{D})\varepsilon^2}{12r}\qquad\text{for sufficiently large $z$},
\]
where $\partial G$ denotes the boundary of $G$ and $\operatorname{dist}(A,B) = \inf\{|a-b|:a\in A,\ b\in B\}$.

Then, recalling that $\frac{1}{T}\int_{A_T}d\tau$ tends to $m(\mathcal{D})$ as $T\to\infty$, gives that the measure of the set of $\tau\in A_T$ satisfying 
\begin{multline*}
\sum_{j=1}^n\left(\iint_{G}\left|L_Q(s+i\gamma_j(\tau),\overline{0};\chi_j)-L_{M_j}(s+i\gamma_j(\tau),\overline{0};\chi_j)\right|^2d\sigma dt\right)\\<\frac{\sqrt{\pi}\operatorname{dist}(\partial G,K)\varepsilon^2}{4}
\end{multline*}
is greater than $\frac{m(\mathcal{D})T}{2}$. Then, using the fact that $|f(s)|\leq\tfrac{||f||}{\sqrt{\pi}\operatorname{dist}(\{s\},\partial{G})}$ for any analytic function $f$ and $s$ lying in the interior of $G$ (see \cite[Chapter III, Lemma 1.1]{G}), we get that the measure of the set of $\tau\in A_T$ satisfying
\[
\max_{1\leq j\leq n}\max_{s\in K}\left|L_Q(s+i\gamma_j(\tau),\overline{0};\chi_j)-L_{M_j}(s+i\gamma_j(\tau),\overline{0};\chi_j)\right|<\frac{\varepsilon}{2}
\]
is greater than $\frac{m(\mathcal{D})T}{2}$, which together with \eqref{eq:fromDenseness0} completes the proof with $v:=\max\{p:p\in \bigcup_{j}M_j\}$.
\end{proof}

\section{Application of  the second moment}

As we described in Section 2, in order to complete the proof of universality we need to show how to approximate shifts of a truncated Euler product by shifts of a corresponding $L$-function. In general a given $L$-function is not well-approximated by a corresponding truncated Euler product in the critical strip with respect to the supremum norm. Nevertheless it is well known that the situation is much easier if we consider the $L^2$-norm, which we use to prove the following result.

\begin{lemma}\label{lem:2kmomentsApp}
Assume that $\chi$ is a Dirichlet character, $a>0$, $\alpha\ne 0$ and $b$ are real numbers, and  $\gamma(t)=\alpha t^a(\log t)^b$. Then, for every $\varepsilon>0$ and  sufficiently large integer $y$, we have
\[
\operatorname{meas}\left\{\tau\in [0,T]: \max_{s\in K}\left|L(s+i\gamma(\tau);\chi)-L_{\{p:p\leq y\}}(s+i\gamma(\tau),\overline{0};\chi)\right|<\varepsilon\right\}>(1-\varepsilon)T
\]
for any compact set $K\subset\{s\in\mathbb{C}:1/2<\sigma<1\}$.
\end{lemma}
\begin{proof}
One can easily observe that it suffices to prove that for sufficiently large $T$ and $y$ we have
\begin{equation}\label{eq:fromCarlson}
\int_1^T\left|L(s+i\gamma(\tau);\chi)-L_{\{p:p\leq y\}}(s+i\gamma(\tau),\overline{0};\chi)\right|^2d\tau <\varepsilon^3T.
\end{equation}
In order to do that we shall prove that for every sufficiently large $X$ we have
\begin{equation}\label{eq:fromCarlsonX}
\int_X^{2X}\left|L(s+i\gamma(\tau);\chi)-L_{\{p:p\leq y\}}(s+
i\gamma(\tau),\overline{0};\chi)\right|^2d\tau <\varepsilon^3X.
\end{equation}

First note that 
\begin{align}\label{eq:GammaToTau}
&\int_X^{2X}\left|L(s+i\gamma(\tau);\chi)-L_{\{p:p\leq y\}}(s+
i\gamma(\tau),\overline{0};\chi)\right|^2d\tau\\
&\qquad\qquad\qquad\qquad\ll X^{1-a}(\log X)^{-b}\nonumber\\
&\qquad\qquad\qquad\qquad\quad\times\int_{X}^{2X}\left|L(s+i\gamma(\tau);\chi)-L_{\{p:p\leq y\}}(s+
i\gamma(\tau),\overline{0};\chi)\right|^2 d\gamma(t).\nonumber
\end{align}
Next, one can easily show that for every $s\in K$ we have
\[
\int_1^{T}\left|L(s+i\tau;\chi)-L_{\{p:p\leq y\}}(s+i\tau,\overline{0};\chi)\right|^2d\tau\ll T
\]
for sufficiently large $T$, so, by Carlson's theorem (see for example Theorem A.2.10 in \cite{KV}), we obtain
\[
\lim_{T\to\infty}\frac{1}{T}\int_1^T\left|L(s+i\tau;\chi)-L_{\{p:p\leq y\}}(s+i\tau,\overline{0};\chi)\right|^2d\tau = \sum_{n\geq y}\frac{c_n}{n^{2\sigma}}
\]
with $c_n=0$ if all primes dividing $n$ are less than $y$, and $c_n=1$ otherwise. Hence, the second factor on the right hand side of \eqref{eq:GammaToTau} is 
\[
\ll \gamma(2X)\sum_{n\geq y}\frac{c_n}{n^{2\sigma}}<\varepsilon^3 X^a\log^b X
\]
for sufficiently large $X$ and $y$, which gives \eqref{eq:fromCarlsonX} and the proof is complete.
\end{proof}

Now we are in the position to prove Theorem \ref{thm:univNonInt}.

\begin{proof}[Proof of Theorem \ref{thm:univNonInt}]
In view of Lemma \ref{lem:finiteProd} and the last lemma it is sufficient to prove that for every finite sets $M_1,\ldots,M_n$ of primes the curve
\[
\gamma(\tau)=\left(\left(\gamma_1(t)\frac{\log p}{2\pi}\right)_{p\in M_1},\ldots,\left(\gamma_n(t)\frac{\log p}{2\pi}\right)_{p\in M_n}\right)
\]
is continuously uniformly distributed mod $1$, where $\gamma_j(t)=\alpha_jt^{a_j}\log^{b_j} t$. By Weyl's criterion we need to prove that
\[
\lim_{T\to\infty}\frac{1}{T}\int_0^T\exp\left(2\pi i\sum_{j=1}^n \gamma_j(t)\left(\sum_{p\in M_j}h_{jp}\frac{\log p}{2\pi}\right)\right)dt = 0
\]
for any non-zero sequence of integers $(h_{jp})$. 

Without loss of generality we can assume that for every $j$ there is at least one $p\in M_j$ such that $h_{jp}\ne 0$. Therefore, $c_j:= \sum_{p\in M_j}h_{jp}\frac{\log p}{2\pi}\ne 0$ for every $1\leq j\leq n$, and again, by Weyl's criterion, it suffices to show that $g(t) = \sum_{j=1}^n c_j\gamma_j(t)$ is continuously uniformly distributed mod $1$ in $\mathbb{R}$. In order to prove it, we shall use \cite[Theorem 9.6]{KN} and show that for almost all $t\in[0,1]$ the sequence $(g(nt))_{n\in\mathbb{N}}$ is uniformly distributed mod $1$ in $\mathbb{R}$ for any real $c_j\ne 0$.

Let $a=\max_{1\leq j\leq n} a_j$, $b=\max\{b_j: 1\leq j\leq n,\ a_j=a\}$ and $j_0$ be an index satisfying $(a_{j_0},b_{j_0}) = (a,b)$. First, let us assume that $a_{j_0}\not\in\mathbb{Z}$, $b_{j_0}\in\mathbb{R}$ or $a_{j_0}\in\mathbb{Z}$, $b_{j_0}<0$. Then it is clear that for every $t\in (0,1)$ the function $g_t(x) = \sum_{j=1}^n c_j\gamma_j(x)$ is $\lceil a\rceil$ times differentiable and $g_t^{(\lceil a\rceil)}(x)\asymp x^{a-\lceil a\rceil}\log^b x$. Hence $g_t^{(\lceil a\rceil)}(x)$ tends monotonically to $0$ as $x\to \infty$ and $x\left|g_t^{(\lceil a\rceil)}(x)\right|\to\infty$ as $x\to\infty$, so, by \cite[Theorem 3.5]{KN}, the sequence $(g_t(n))=(g(nt))$, $n=1,2,\ldots$, is uniformly distributed mod $1$.

The case $a_{j_0}\in\mathbb{N}$ and $b_{j_0}>1$ is very similar, since $g_t^{(\lceil a\rceil+1)}(x) \asymp \frac{\log^{b-1} x}{x}$.

Finally, if $a_{j_0}\in\mathbb{N}$ and $b_{j_0}=0$, we see that $\lim_{x\to\infty} g_t^{(a)}(x)\to t^a a! c_{j_0}\alpha_{j_0}$, which is irrational for almost all $t\in [0,1]$. Therefore, \cite[Chapter 1, Section 3]{KN} (in particular see \cite[Exercise 3.7, p. 31]{KN}) shows that the sequence $(g_t(n))=(g(nt))$, $n=1,2,\ldots$, is uniformly distributed mod $1$ for almost all $t\in [0,1]$, and the proof is complete.
\end{proof}

\section{Discrete version}

In this section we deal with a discrete version of Theorem \ref{thm:univNonInt}. Let us start with the following discrete analogue of Lemma \ref{lem:finiteProd}.

\begin{lemma}\label{lem:finiteProdDiscrete}
Assume that the functions $\gamma_j:\mathbb{R}\to\mathbb{R}$, $1\leq j\leq n$, and $\mathcal{P}_j\subset\mathbb{P}$ are minimal sets such that the curve
\[
\gamma(k)=\left(\left(\gamma_1(k)\frac{\log p}{2\pi}\right)_{p\in M_1},\ldots,\left(\gamma_n(k)\frac{\log p}{2\pi}\right)_{p\in M_n}\right)
\]
is uniformly distributed mod $1$ for any finite sets of primes $M_j\subset\mathbb{P}\setminus\mathcal{P}_j$, $1\leq j\leq n$. Moreover, let $\chi_1,\ldots,\chi_n$ be arbitrary Dirichlet characters, $K\subset\{s\in\mathbb{C}:1/2<\sigma<1\}$ be a compact set with connected complement and $f_1,\ldots,f_n$ continuous non-vanishing function on $K$, which are analytic in the interior of $K$. Then, for every $\varepsilon>0$ and every finite sets $\mathcal{A}_j$ with $ \mathcal{P}_j\subset\mathcal{A}_j\subset\mathbb{P}$, there is $v>0$ such that for every $y>v$ we have
\[
\sharp\left\{2\leq k\leq N: \begin{matrix}
\max_{1\leq j\leq n}\max_{s\in K}\left|L_{\{\mathcal{A}_j\not\ni p:p\leq y\}}(s+i\gamma_j(k),\overline{0};\chi_j) - f_j(s)\right|<\varepsilon\\
\max_{1\leq j\leq n}\max_{p\in \mathcal{A}_j\setminus\mathcal{P}_j}\left\Vert\gamma_j(k)\frac{\log p}{2\pi}\right\Vert<\varepsilon
\end{matrix}\right\}>cN
\]
with suitable constant $c>0$, which does not depend on $y$.
\end{lemma}
\begin{proof}
The proof closely follows the proof of Lemma \ref{lem:finiteProd}, therefore we will be rather sketchy. 

As in the proof of Lemma \ref{lem:finiteProd} we use Mergelyan's theorem and Lemma \ref{lem:denseness} to find the set $G$, real numbers $\theta_{jp}$ for $p\in\mathbb{P}\setminus \mathcal{A}_j$, $1\leq j\leq n$, finite sets of primes $M_j$, $1\leq j\leq n$,  containing $\{p\in\mathbb{P}\setminus\mathcal{A}_j:p\leq z\}$ and satisfying 
\[
\max_{1\leq j\leq n}\max_{s\in\overline{G}}\left|L_{M_j}(s,(\theta_{jp})_{p\in M_j};\chi_j)-f_j(s)\right|<\frac{\varepsilon}{2}.
\]

Moreover, we put $\theta_{jp}=0$ for $p\in\mathcal{A}_j\setminus\mathcal{P}_j$ and $Q_j:=\{p\in\mathbb{P}\setminus \mathcal{P}_j: p<y\}\supset \bigcup_{1\leq j\leq n}M_j$ and then define the set $\mathcal{D}$ and $\delta>0$ as in the proof of Lemma \ref{lem:finiteProd}.

Let us notice that, in view of the choice of the sets $\mathcal{P}_j$, $\mathcal{A}_j$ and $M_j$, the set $A$ of positive integers $k$ satisfying
\[
\max_{1\leq j\leq n}\max_{p\in M_j}\left\Vert\gamma_j(k)\frac{\log p}{2\pi}-\theta_{jp}\right\Vert<\delta\qquad \max_{1\leq j\leq n}\max_{p\in \mathcal{A}_j\setminus\mathcal{P}_j}\left\Vert\gamma_j(k)\frac{\log p}{2\pi}\right\Vert<\varepsilon
\]
has a positive density equal to $m(\mathcal{D})$ and
\begin{equation}
\max_{1\leq j\leq n}\max_{s\in\overline{G}}\left|L_{M_j}(s+i\gamma_j(k),\overline{0};\chi_j)-f_j(s)\right|<\varepsilon,\qquad\text{for $k\in A$}.
\end{equation}

Now, let us define $A_N=A\cap [2,N]$ and consider
\[
S_j = \frac{1}{N}\sum_{k\in A_N}\iint_{G}\left|L_{Q_j\setminus\mathcal{A}_j}(s+i\gamma_j(k),\overline{0};\chi_j)-L_{M_j}(s+i\gamma_j(k),\overline{0};\chi_j)\right|^2d\sigma dt.
\]
Since $\gamma(k)$ is  uniformly distributed mod $1$ and $Q_j\setminus (M_j\cup \mathcal{A}_j)$ contains only primes greater than $z$, we obtain from \cite[Theorem 6.1]{KN}) that
\begin{align*}
&\lim_{N\to\infty} \frac{1}{N}\sum_{k\in A_N}\left|L_{Q_j\setminus\mathcal{A}_j}(s,(\gamma_j(k)\frac{\log p}{2\pi});\chi_j)-L_{M_j}(s,(\gamma_j(k)\frac{\log p}{2\pi});\chi_j)\right|^2\\
&\qquad\qquad = \idotsint\displaylimits_{\mathcal{D}}|L_{M_j}(s,\overline{\omega};\chi_j)|^2|L_{Q_j\setminus (M_j\cup\mathcal{A}_j)}(s,\overline{\omega};\chi_j)-1|^2d\omega\\
&\qquad\qquad <\frac{\sqrt{\pi}\operatorname{dist}(\partial G,K)m(\mathcal{D})\varepsilon^2}{12r}
\end{align*}

Then, again uniform distribution mod $1$ of $\gamma(k)$ gives that $\frac{1}{N}\sharp A_k$ tends to $m(\mathcal{D})$ as $N\to\infty$. Hence the number of $k\in A_N$ satisfying 
\begin{multline*}
\sum_{j=1}^n\left(\iint_{G}\left|L_{Q_j\setminus\mathcal{A}_j}(s+i\gamma_j(k),\overline{0};\chi_j)-L_{M_j}(s+i\gamma_j(k),\overline{0};\chi_j)\right|^2d\sigma dt\right)\\<\frac{\sqrt{\pi}\operatorname{dist}(\partial G,K)\varepsilon^2}{4}
\end{multline*}
is greater that $\frac{m(\mathcal{D})N}{2}$. Then, the proof is complete as in the proof of Lemma~\ref{lem:finiteProd}.
\end{proof}

The next proposition is a discrete version of Lemma \ref{lem:2kmomentsApp} and its proof relies on Carlson's theorem and the following Gallagher's lemma
\begin{lemma}[Gallagher] Let $T_0$ and $T\geq \delta>0$ be real numbers, and $A$ be a finite subset of $[T_0+\delta/2,T+T_0-\delta/2]$. Define $N_\delta(x) = \sum_{t\in A,\ |t-x|<\delta} 1$ and assume that $f(x)$ is a complex continuous function on $[T_0,T+T_0]$ continuously differentiable on $(T_0,T+T_0)$. Then
\begin{align*}
\sum_{t\in A}N_\delta^{-1}(t)|f(t)|^2&\leq \frac{1}{\delta}\int_{T_0}^{T+T_0}|f(x)|^2 dx\\
&\quad +\left(\int_{T_0}^{T+T_0}|f(x)|^2 dx\int_{T_0}^{T+T_0}|f'(x)|^2 dx\right)^{1/2}.
\end{align*}
\end{lemma}
\begin{proof}
This is Lemma 1.4 in \cite{M}.
\end{proof}

\begin{proposition}\label{prop:meanValueGeneralDiscrete}
Assume that $\chi$ is a Dirichlet character, $a>0$, $\alpha\ne 0$ and $b$ are real numbers, and  $\gamma(t)=\alpha t^a(\log t)^b$. Then, for every $\varepsilon>0$ and  sufficiently large integer $y$, we have
\[
\sharp\left\{2\leq k\leq N: \max_{s\in K}\left|L(s+i\gamma(k);\chi)-L_{\{p:p\leq y\}}(s+i\gamma(k),\overline{0};\chi)\right|<\varepsilon\right\}>(1-\varepsilon)N
\]
for any compact set $K\subset\{s\in\mathbb{C}:1/2<\sigma<1\}$.
\end{proposition}
\begin{proof}
Let us apply Gallagher's lemma for $f(x) = L(s+i\gamma(x);\chi)-L_{\{p:p\leq y\}}(s+i\gamma(x),\overline{0};\chi)$ with $\delta=1/2$, $T_0=1$, $T=N$ and $A=\{2,3,\ldots,N\}$. Then $N_\delta(t)=1$ for every $t\in A$, so
\begin{align*}
&\frac{1}{N}\sum_{k=2}^N|L(s+i\gamma(k);\chi)-L_{\{p:p\leq y\}}(s+i\gamma(k),\overline{0};\chi)|^2\\
&\qquad\qquad\ll \frac{1}{N}\int_{1}^{N+1}|L(s+i\gamma(t);\chi)-L_{\{p:p\leq y\}}(s+i\gamma(t),\overline{0};\chi)|^2 dt\\
&\qquad\qquad\quad+\left(\frac{1}{N}\int_{1}^{N+1}|L(s+i\gamma(t);\chi)-L_{\{p:p\leq y\}}(s+i\gamma(t),\overline{0};\chi)|^2 dt\right.\\
&\qquad\qquad\quad\times\left.\frac{1}{N}\int_{1}^{N+1}|L'(s+i\gamma(t);\chi)-L'_{\{p:p\leq y\}}(s+i\gamma(t),\overline{0};\chi)|^2 dt\right)^{1/2}.
\end{align*}
Then, as we observed in the proof of Lemma \ref{lem:2kmomentsApp}, Carlson's theorem gives \eqref{eq:fromCarlson}. Moreover, Cauchy's integration formula implies the truth of \eqref{eq:fromCarlson} for $L'$ as well. Therefore, we see that the right hand side of the above inequality is $<\varepsilon^3$ for sufficiently large $N$ and $y$, and the proof is complete.
\end{proof}

\begin{proof}[Proof of Theorem \ref{thm:univDiscrete}]
First we shall use Lemma \ref{lem:finiteProdDiscrete}, so let us define the sets $\mathcal{A}_j$ and $\mathcal{P}_j$ for $j=1,2,\ldots,n$. If $\gamma_j(t)=\alpha_jt^{a_j}\log^{b_j}t$ with $a_j\notin\mathbb{Z}$ or $b_j\ne 0$, then the proof is essentially the same as in the continuous case, so we just take $\mathcal{P}_j=\mathcal{A}_j=\emptyset$. 

The more delicate situation is when $a_j\in\mathbb{N}$ and $b_j=0$ for some $1\leq j\leq n$, since the sequence $(\alpha_j k^{a_j}\sum_{p\in M_j} h_{jp}\tfrac{\log p}{2\pi})_{k\in \mathbb{N}}$ is uniformly distributed mod $1$ only if $\alpha_j \sum_{p\in M_j} h_{jp}\tfrac{\log p}{2\pi}$ is irrational. In order to overcome this obstacle we define the sets $\mathcal{P}_j$ and $\mathcal{A}_j$ as follows. Let $m^*_{j}$ be the smallest positive integer such that $\exp(2\pi m^*_{j}/\alpha_j)\in \mathbb{Q}$. Note that for every $m\in\mathbb{Z}$ satisfying $\exp(2\pi m/\alpha_j)\in \mathbb{Q}$ we have $m^*_j|m$. Assume that 
\begin{equation}\label{eq:factorization}
\exp(2\pi m^*_{j}/\alpha_j)  = \prod_{p\in \mathcal{A}_j} p^{k_{jp}}
\end{equation}
for some integers $k_{jp}\ne 0$ and some finite set of primes $\mathcal{A}_j$. Moreover, let $p_j^*$ be the least prime number in the set $A_j$ and put $\mathcal{P}_j = \{p_j^*\}$. Let us notice that the choice of $\mathcal{P}_j$ implies that it is a minimal set such that $\alpha_j \sum_{p\in M_j} h_{jp}\tfrac{\log p}{2\pi}\notin\mathbb{Q}$ for every non-zero sequence of integers $h_{jp}$ and a finite set of primes $M_j$ disjoint to $\mathcal{P}_j$,  since otherwise there exist integers $m,l$ such that $\exp(2\pi m/\alpha_j)=\prod_{p\in M_j}p^{lh_{jp}}\in\mathbb{Q}$, which, by the definition of $m^*_j$, is a power of $\prod_{p\in\mathcal{A}_j}p^{k_{jp}}$, and we get a contradiction. 

Hence, arguing similarly to the proof of Theorem \ref{thm:univNonInt} (see \cite[Theorem 3.5 and Exercise 3.7, p.31]{KN}), the curve
\[
\gamma^*(t) = \left(\left(\gamma_j(q^*t)\frac{\log p}{2\pi k_{jp_j^*}}\right)_{p\in\mathcal{A}_j\setminus\mathcal{P}_j},\left(\gamma_j(q^*t)\frac{\log p}{2\pi}\right)_{p\in M_j}\right)_{1\leq j\leq n},
\]
is uniformly distributed mod $1$ for every finite sets of primes $M_j$ disjoint to $\mathcal{P}_j$, where $q^*$ is the least common multiple of all $k_{jp_j^*}$ for $j$ satisfying $a_j\in\mathbb{Z}$ and $b_j=0$. If $a_j\not\in\mathbb{Z}$ or $b_j\ne 0$ for all $j=1,2,\ldots,n$, then $q^*=1$.

Therefore, applying Lemma \ref{lem:finiteProdDiscrete} for
\[
f^*_j(s)=\prod_{p\in A_j}\left(1-\frac{\chi_j(s)}{p^s}\right)f_j(s)
\]
instead of $f_j(s)$, gives that the number of integers $k\in [2,N]$ satisfying
\[
\max_{1\leq j\leq n}\max_{s\in K}\left|L_{\{\mathcal{A}_j\not\ni p:p\leq y\}}(s+i\gamma_j(q^*k),\overline{0};\chi_j) - f^*_j(s)\right|<\varepsilon
\]
\[
\max_{1\leq j\leq n}\max_{p\in \mathcal{A}_j\setminus\mathcal{P}_j}\left\Vert\gamma_j(q^*k)\frac{\log p}{2\pi k_{jp_j^*}}\right\Vert<\varepsilon
\]
is at least $cN$. The second inequality together with \eqref{eq:factorization} gives that
\[
\max_{1\leq j\leq n}\max_{p\in \mathcal{A}_j\setminus\mathcal{P}_j}\left\Vert\gamma_j(q^*k)\frac{\log p}{2\pi}\right\Vert\ll\varepsilon
\]
and, for every $j$ satisfying $a_j\in\mathbb{Z}$, $b_j=0$,
\[
\left\Vert\gamma_j(q^*k)\frac{\log p_j^*}{2\pi}\right\Vert=\left\Vert\frac{m_j^*}{\alpha_jk_{jp_j^*}}\gamma_j(q^*k)+\sum_{p\in\mathcal{A}_j\setminus\mathcal{P}_j}\gamma_j(q^*k)\frac{\log p}{2\pi k_{jp_j^*}}\right\Vert\ll \varepsilon,
\]
since $\gamma_j(q^*k)/\alpha_j=(q^*)^{a_j}k^{a_j}$ is a multiple of $k_{jp_j^*}$ by the definition of $q^*$.

Thus
\[
\prod_{p\in \mathcal{A}_j}\left(1-\frac{\chi_j(p)}{p^{s+i\gamma_j(q^*k)}}\right)^{-1}f_j(s)
\]
approximates $f_j^*(s)$ uniformly on $K$, and hence
\[
\max_{1\leq j\leq n}\max_{s\in K}\left|L_{\{ p:p\leq y\}}(s+i\gamma_j(q^*k),\overline{0};\chi_j) - f_j(s)\right|\ll \varepsilon.
\]
Moreover, by replacing $q^*k$ by $k$, one can easily observe that the number of integers $k\in [2,N]$ satisfying
\[
\max_{1\leq j\leq n}\max_{s\in K}\left|L_{\{ p:p\leq y\}}(s+i\gamma_j(k),\overline{0};\chi_j) - f_j(s)\right|\ll \varepsilon
\]
is at least $cN/q^*$, which together with Proposition \ref{prop:meanValueGeneralDiscrete} and Lemma \ref{lem:2kmomentsApp}, complete the proof.
\end{proof}

\paragraph{\bf Acknowledgment} The author would like to cordially thank to Professor Kohji Matsumoto for his valuable comments and  suggestions.

\end{document}